\documentclass[12pt]{amsart}
\usepackage{lipsum}
\usepackage{amssymb}
\usepackage{enumerate}
\usepackage{amsthm}
\usepackage{tikz-cd}
\usepackage{amsmath}
\usepackage[mathscr]{euscript}
\usepackage[utf8]{inputenc}
\usepackage[english]{babel}
\usepackage{hyperref}
\hypersetup{
    colorlinks=true,
    linkcolor=blue,
    citecolor=black,
    filecolor=red,      
    urlcolor=violet,
}

\makeatletter
\@namedef{subjclassname@2020}{%
  \textup{2020} Mathematics Subject Classification}
\makeatother
\newtheorem{thm}{Theorem}[section]
\newtheorem{lem}[thm]{Lemma}

\newtheorem{defi}[thm]{Definition}

\newtheorem{qn}[thm]{Question}

\DeclareMathOperator{\rk}{{rk}}

\begin{document}
\baselineskip=16pt


\author{Snehajit Misra}

\address{Department of Mathematics and Computing, Indian Institute of Technology (Indian School of Mines) IIT (ISM) Dhanbad - 826004, Jharkhand, India.}
\email[Snehajit Misra]{misra08@gmail.com, snehajitm@iitism.ac.in}

\begin{abstract}
 In this article, we investigate the instability of syzygy bundles corresponding to globally generated vector bundles on smooth irreducible projective surfaces under change of polarization.
\end{abstract}

\title{On instability of Syzygy Bundles}

\maketitle

\section{Introduction}
Let $X$ be a smooth irreducible projective variety of dimension $n$ on an algebraically closed field $k$, and let $E$ be a globally generated vector bundle of rank $r$ on $X$. 
The syzygy bundle or the kernel bundle associated with $E$, denoted by $M_{E}$ is defined as the kernel of the evaluation map $$ev : H^0(X,E)\otimes \mathcal{O}_X\longrightarrow E.$$ Thus we have the following short exact sequence:
\begin{center}
 $0\longrightarrow M_{E}\longrightarrow H^0(X,E) \otimes \mathcal{O}_X\longrightarrow E\longrightarrow 0.$
\end{center}
The question of (semi)stability or instability of $M_{E,V}$ is a long standing problem (see Definition \ref{defi3.1} for semistability) as these vector bundles arise in a variety of geometric and algebraic
problems.  This has major implications in higher rank Brill-Noether theory.

For example, the stability of these bundles $M_E$ has been studied in \cite{B94} when $E$ is a globally generated vector bundle of arbitrary rank $r$ on a smooth complex projective curve $C$. In particular, it is shown that if $E$ is semistable globally generated vector bundle and $\mu(E)\geq 2g(C)$, where $g(C)$ is the genus of the curve $C$, then $M_E$ is semistable. Furthermore, if $E$ is a globally generated stable vector bundle on a smooth curve $C$ under the hypothesis $\mu(E)\geq 2g(C)$, then $M_E$ is stable unless $\mu(E) = 2g(C)$ and $C$ is hyperelliptic or $\Omega_C\hookrightarrow E$. 

In \cite{ELM13}, the authors
study the stability of $M_L$  when $L$ is a very ample line bundle on a smooth complex projective surface $X$. In fact, they showed the following: if we fix an ample divisor $H$ on $X$ and an arbitrary divisor $P$ on the algebraic surface $X$ and given a large integer $d$, set
$$L_d = dH + P$$
and write $M_d = M_{L_d}$, then for sufficiently large $d$, the syzygy bundle $M_d$ is slope stable with respect
to ample divisor $H$. Ein-Lazarsfield-Mustopa  also conjectured the following :  Let $X$ be a smooth projective variety of dimension $n$, and define $M_d$ as above. Then $M_d$ is $H$-stable for every $d\gg 0$. The authors also asked the question about the behavior regarding stability of $M_L$ under change of polarization.
This above mentioned conjecture due to Lazarsfeld et al. is solved in \cite{R24} (see \cite[Theorem 4.3, Corollary 4.4]{R24}.) There are several other works, in the years, where stability of syzygy
bundles has been used, with various purposes. See also \cite{B08},\cite{C11},
\cite{FM23},\cite{DHJS25} for related works.

\vspace{2mm}

Though there has been a considerable amount of literature on the (semi)stability of the syzygy bundle corresponding to globally generated line bundles, not much is known for the case of syzygy bundles corresponding to  higher rank globally generated vector bundles $E$ on higher dimensional smooth varieties.

Recently, the stability of $M_E$ has been studied for stable vector bundles $E$ of higher rank on a smooth projective surface $X$.
For any vector
bundle $E$ on a smooth polarized variety $(X,H)$ and $m > 0$, we denote $E(m) := E \otimes \mathcal{O}_X(mH)$. The authors in \cite{BP23} prove that if $E$ is an $H$-stable vector bundle on a smooth projective surface $X$, then $M_{E(m)}$ is $H$-stable for $m\gg 0$.
However no such effective result is known for stability of syzygy bundles corresponding to stable globally generated vector bundles on smooth surfaces.
\vspace{2mm}
 More generally, a sufficient condition is given for semistability of syzygy bundles $M_E$ corresponding to ample and globally generated vector bundle $E$ on smooth surfaces in \cite[Theorem 3.2]{MN26}.

 As a corollary to the above mentioned theorem (i.e. \cite[Theorem 3.2]{MN26}), the authors proved the following result extending the result in \cite{BP23} to semistable globally generated vector bundles on smooth surfaces : If $E$ is an $H$-semistable vector bundle on a smooth projective surface $X$, then $M_{E(m)}$ is $H$-stable for $m\gg 0$. In this article, we try to answer the following question :

 \begin{qn}
Let $X$ be a smooth projective variety and $E$ be a globally generated vector bundle on $X$. How does the behavior of (semi)stability changes of the corresponding  syzygy bundles $M_E$ under change of polarization? 
 \end{qn}

 The above question has been answered in \cite{DHJS25} for syzygy bundles $M_L$ corresponding to ample line bundles $L$ on toric surfaces. In particular, it has been shown that if $L=\mathcal{O}_X(D)$ is an ample line bundle on a smooth toric surface $X$ apart from $\mathbb{P}^2$ or $\mathbb{P}^1\times \mathbb{P}^1$, then there exists a polarization $A$ such that $M_{\mathcal{O}_X(dD)}$ is not stable with respect to $A$  for $d\gg 0$.  In this article, our main result is the following: 
 \begin{thm}
    Let $X$ be a smooth surface of Picard number at least 3 having an effective cone generated by curves $C_1,C_2,\cdots, C_m$ with mutual intersection multiplicity of at most 1. Then for every vector bundle $E$ with its determinant bundle $D=\det(E)$ ample, there exists an ample polarisation $A$ such that for $d\gg 0$ the syzygy bundle $M_{E(d)}$ is not stable with respect to $A$.
    \end{thm}

    We closely follow the ideas in \cite{DHJS25}.

\section{Notations and Conventions}
We work over the field of complex numbers $\mathbb{C}$. Given a coherent
sheaf $\mathcal{G}$ on a variety $X$, we write $h^i(\mathcal{G})$ to denote the dimension of the $i$-th cohomology group $H^i(X,\mathcal{G})$. The sheaf $K_X$ will denote the canonical sheaf on $X$. In this article, all the varieties are assumed to be irreducible.  Also, throughout this article, (semi)stability means slope (semi)stability. The $i$-th Chern class of a vector bundle $E$ will be denoted by $c_i(E)$. The rank of a vector bundle $E$ will be denoted by rk$(E)$.
\section{Main results}
Let $X$ be a smooth complex projective variety of dimension $n$ with a fixed ample line bundle $H$ on it.
For a non-zero vector bundle $V$ of rank $r$ on $X$, the $H$-slope of $V$ is defined as
\begin{align*}
\mu_H(V) := \frac{c_1(V)\cdot H^{n-1}}{r} \in \mathbb{Q},
\end{align*}
where $c_1(V)$ denotes the first Chern class of $V$.
\begin{defi}\label{defi3.1}
A vector bundle $V$ on $X$ is said to be $H$-semistable (respectively stable) if $\mu_H(W) \leq \mu_H(V)$ (respectively $\mu_H(W) < \mu_H(V)$) for all subsheaves $W$ of $V$. A vector bundle $V$ on $X$ is called $H$-unstable if it is not $H$-semistable.
\end{defi}
See \cite{HL10} for more details about semistability.
\vspace{1mm}

Recall the short exact sequence for a globally generated vector bundle $E$.
\begin{center}
 $0\longrightarrow M_{E}\longrightarrow H^0(X,E)\otimes \mathcal{O}_X\longrightarrow E\longrightarrow 0.$
\end{center}

Hence, the $H$-slope of a syzygy bundle $M_E$ associated to a globally generated vector bundle $E$ is as follows :
$$\mu_H(M_E) = \frac{c_1(M_E)\cdot H^{n-1}}{\rk(M_E)} = \frac{-c_1(E)\cdot H^{n-1}}{h^0(E)-\rk(E)}.
$$
\begin{lem}\label{lem1}
    Let $E$ be a vector bundle on a smooth projective variety $X$, and $F$ be a subbundle of $E$. If both $E$ and $F$ are globally generated, then $M_F$ is a subbundle of $M_E$.
    \begin{proof}
        The result follows from the following commutative diagram : 
         \begin{center}
\begin{tikzcd} 
0 \arrow[r, " "] &M_F\arrow[r, " "] \arrow[d, " "] 
& \arrow[r, " "] H^0(X,F) \otimes \mathcal{O}_X \arrow[r, " "] \arrow[d," "] 
& F \arrow[r, "  "] \arrow[d," "] & 0  \\
0 \arrow[r, " "] & M_E  \arrow[r, " " ] \arrow[d, " "] & H^0(X,E) \otimes \mathcal{O}_X \arrow[r, " "] \arrow[d," "]
&E \arrow[r, " "] \arrow[d," "] & 0 \\
 & 0 & 0 & 0 
\end{tikzcd}
\end{center}
    \end{proof}
\end{lem}

\begin{thm}\label{thm-main}
    Let $(X,A)$ be a polarized surface and $E$ be a vector bundle of rank $r$ on $X$ such that $D=\det(E)$ is ample. Let $S$ be an effective divisor and $d$ be a positive integer such that $$E(d) = E\otimes \mathcal{O}_X(dD),$$ and $E(d)$ and $E(d)\otimes \mathcal{O}_X(-S)$ are globally generated. Then 
    \begin{enumerate}
 \item If $$(2r-1)(D^2)(S\cdot A)  - 2(D\cdot A)(D\cdot S) >0,$$
        then $$\mu_A\bigl(M_{E(d)\otimes \mathcal{O}_X(-S)}\bigr) > \mu_A\bigl(M_{E(d)}\bigr)$$ for $d\gg0$.
        \item If $$(2r-1)(D^2)(S\cdot A)  - 2(D\cdot A)(D\cdot S) \geq 0$$ and
        $$d>-\dfrac{B}{C} $$ where $$B = r\bigl[(2r-1)(S\cdot A) (D^2)-2(D\cdot A)(D\cdot S)\bigr]+\dfrac{r^2}{2}\bigl[\bigl( K_X\cdot S)+S^2\bigr)(D\cdot A)-(D\cdot K_X)(S\cdot A)\Bigr]$$\\
and $$C=\bigl[\dfrac{r}{2}(S\cdot A)(D^2) -(D\cdot A)(D\cdot S)\bigr]+\dfrac{r}{2}\bigl[\bigl( K_X\cdot S)+S^2\bigr)(D\cdot A)-(D\cdot K_X)(S\cdot A)\Bigr] + $$
\hspace{3cm} $r^2\bigl( \chi(\mathcal{O}_X)-1\bigr) + r\,c_2(E)(S\cdot A);$\\

then  $$\mu_A\bigl(M_{E(d)\otimes \mathcal{O}_X(-S)}\bigr) > \mu_A\bigl(M_{E(d)}\bigr).$$
    \end{enumerate}
    \begin{proof}
        Recall that for any globally generated vector bundle $E$ of rank $r$ on $X$, we have $$\mu_A(M_E) = -\dfrac{c_1(E)\cdot A}{h^0(E)-r}$$
        Thus 
        \begin{align}\label{s1}
        \mu_A\bigl(M_{E(d)}\bigr) = - \dfrac{c_1(E(d))\cdot A}{h^0(E(d))-r}
        \end{align}
        and 
        \begin{align}\label{s2}
            \mu_A\bigl(M_{E(d)\otimes \mathcal{O}_X(-S)}\bigr) =-\dfrac{c_1\bigl(E(d)\otimes \mathcal{O}_X(-S)\bigr)\cdot A}{h^0\bigl(E(d)\otimes \mathcal{O}_X(-S)\bigr)-r}
        \end{align}
        Therefore 
        $$  \mu_A\bigl(M_{E(d)\otimes \mathcal{O}_X(-S)}\bigr) -  \mu_A\bigl(M_{E(d)}\bigr) \\
        = \dfrac{c_1(E(d))\cdot A}{h^0(E(d))-r} - \dfrac{c_1\bigl(E(d)\otimes \mathcal{O}_X(-S)\bigr)\cdot A}{h^0\bigl(E(d)\otimes \mathcal{O}_X(-S)\bigr)-r} $$
        \begin{align}\label{s3}
=\dfrac{\Bigl[\bigl\{c_1(E(d))\cdot A\bigr\}\bigl\{ h^0\bigl(E(d)\otimes \mathcal{O}_X(-S)\bigr)-r \bigr\}\Bigr] - \Bigr[\bigl\{c_1\bigl(E(d)\otimes \mathcal{O}_X(-S)\bigr)\cdot A\bigr\}\bigl\{ h^0(E(d))-r  \bigr\}\Bigr]}{\bigl\{ h^0(E(d))-r\bigr\}\bigl\{ h^0\bigl(E(d)\otimes \mathcal{O}_X(-S)\bigr)-r\bigl\}}
        \end{align}
        Hence $$ \mu_A\bigl(M_{E(d)\otimes \mathcal{O}_X(-S)}\bigr) -  \mu_A\bigl(M_{E(d)}\bigr) > 0 $$
        if and only if 
        \begin{align}\label{s4}
\Bigr[\bigl\{c_1(E(d))\cdot A\bigr\}\bigl\{ h^0\bigl(E(d)\otimes \mathcal{O}_X(-S)\bigr)-r \bigr\}\Bigr] - \Bigr[\bigl\{c_1\bigl(E(d)\otimes \mathcal{O}_X(-S)\bigr)\cdot A\bigr\}\bigl\{ h^0(E(d))-r  \bigr\}\Bigr] >0.
\end{align}
Now\begin{align*}
    & c_1\bigl(E(d)\otimes \mathcal{O}_X(-S)\bigr)\\
    &  = c_1(E(d))+rc_1(\mathcal{O}_X(-S)) 
\end{align*} 
and \begin{align*}
    & c_2\bigl(E(d)\otimes \mathcal{O}_X(-S)\bigr)\\
    & = c_2(E(d))+(r-1)c_1(E(d))\cdot c_1(\mathcal{O}_X(-S))+\binom{r}{2}c_1(\mathcal{O}_X(-S))\cdot c_1(\mathcal{O}_X(-S))
    \end{align*}
    \end{proof}
\end{thm}
Choose $d\gg 0$ so that $h^i(E(d))=0$ for $i=1,2$.

Now by Riemann-Roch theorem for vector bundles on surfaces, we have
$$h^0(E(d)) = r\chi(\mathcal{O}_X)+\dfrac{1}{2}c_1(E(d))\cdot\bigr(c_1(E(d))-K_X\bigr) -c_2(E(d)),$$
and\\

$h^0\bigl( E(d)\otimes \mathcal{O}_X(-S)\bigr)\\$

$=  r\chi(\mathcal{O}_X)+\dfrac{1}{2}c_1\bigl(E(d)\otimes \mathcal{O}_X(-S)\bigr)\cdot\bigr(c_1\bigl(E(d)\otimes \mathcal{O}_X(-S)\bigr)-K_X\bigr) -c_2(E(d)\otimes \mathcal{O}_X(-S))$.\\

$= r\chi(\mathcal{O}_X) + \dfrac{1}{2}\bigl\{c_1(E(d))+rc_1(\mathcal{O}_X(-S))\bigr\}\bigl\{(c_1(E(d))-K_X)+rc_1(\mathcal{O}_X(-S))\bigr\}\\$

$\hspace{10mm} -c_2(E(d))-(r-1)c_1(E(d))\cdot c_1(\mathcal{O}_X(-S))-\binom{r}{2}c_1(\mathcal{O}_X(-S)\cdot c_1(\mathcal{O}_X(S)).$
\vspace{5mm}

Let $m=c_1(E(d))\cdot A$ . \\

Therefore, 
\vspace{2mm}

$\bigl\{c_1(E(d))\cdot A\bigr\}\bigl\{ h^0\bigl(E(d)\otimes \mathcal{O}_X(-S)\bigr)-r \bigr\} = \bigl\{ h^0\bigl(E(d)\otimes \mathcal{O}_X(-S)\bigr)-r \bigr\} m \\$

$= r\bigl(\chi(\mathcal{O}_X)-1\bigr)m + \dfrac{1}{2}\bigl[c_1(E(d))\cdot \bigl(c_1(E(d))-K_X\bigr)\bigr]m + \dfrac{r}{2}\bigl[c_1(E(d))\cdot c_1(\mathcal{O}_X(-S))\bigr]m \\$

$ \hspace{2mm} + \dfrac{r}{2}\bigl[(c_1(E(d))-K_X)\cdot c_1(\mathcal{O}_X(-S))\bigr]m + \dfrac{r^2}{2}\bigl[ c_1(\mathcal{O}_X(-S)\cdot c_1(\mathcal{O}_X(-S)\bigr]m$\\

$ \hspace{2mm} -c_2(E(d)m - (r-1)\bigl[ c_1(E(d))\cdot c_1(\mathcal{O}_X(-S))\bigr]m - \binom{r}{2}\bigl[c_1(\mathcal{O}_X(-S))\cdot c_1(\mathcal{O}_X(-S))\bigr]m.$

\hspace{3mm}

Similarly, $\bigl\{c_1\bigl(E(d)\otimes \mathcal{O}_X(-S)\bigr)\cdot A\bigr\}\bigl\{ h^0(E(d))-r  \bigr\}$\\

$=\bigl\{c_1(E(d))\cdot A + rc_1(\mathcal{O}_X(-S))\cdot A\bigr\}\bigl\{ r\bigl(\chi(\mathcal{O}_X)-1\bigr)+\dfrac{1}{2}c_1(E(d))\cdot (c_1(E(d))-K_X) - c_2(E(d))\bigr\}$\\

$=\bigl\{m + rc_1(\mathcal{O}_X(-S))\cdot A\bigr\}\bigl\{ r\bigl(\chi(\mathcal{O}_X)-1\bigr)+\dfrac{1}{2}c_1(E(d))\cdot (c_1(E(d))-K_X) - c_2(E(d))\bigr\}$\\

$=r(\chi(\mathcal{O}_X)-1)m + \dfrac{1}{2}\bigl[c_1(E(d))\cdot \bigl(c_1(E(d))-K_X\bigr)\bigr]m - c_2(E(d))m$ \\ 

$\hspace{5mm} 
 + \bigr\{rc_1(\mathcal{O}_X(-S))\cdot A \bigr\}\bigl\{ r\bigl(\chi(\mathcal{O}_X)-1\bigr)+\dfrac{1}{2}c_1(E(d))\cdot (c_1(E(d))-K_X) - c_2(E(d))\bigr\}$\\

\hspace{5mm}

Thus $\Bigr[\bigl\{c_1(E(d))\cdot A\bigr\}\bigl\{ h^0\bigl(E(d)\otimes \mathcal{O}_X(-S)\bigr)-r \bigr\}\Bigr] - \Bigr[\bigl\{c_1\bigl(E(d)\otimes \mathcal{O}_X(-S)\bigr)\cdot A\bigr\}\bigl\{ h^0(E(d))-r  \bigr\}\Bigr]$\\

\hspace{2mm}

$= \dfrac{r}{2}\bigl[c_1(E(d))\cdot c_1(\mathcal{O}_X(-S))\bigr]m + \dfrac{r}{2}\bigl[(c_1(E(d))-K_X)\cdot c_1(\mathcal{O}_X(-S))\bigr]m $\\

$+ \dfrac{r^2}{2}\bigl[ c_1(\mathcal{O}_X(-S)\cdot c_1(\mathcal{O}_X(-S)\bigr]m 
 - (r-1)\bigl[ c_1(E(d))\cdot c_1(\mathcal{O}_X(-S))\bigr]m - \binom{r}{2}\bigl[c_1(\mathcal{O}_X(-S))\cdot c_1(\mathcal{O}_X(-S))\bigr]m $\\

$ - \bigr\{rc_1(\mathcal{O}_X(-S))\cdot A \bigr\}\bigl\{ r(\chi(\mathcal{O}_X)-1)+\dfrac{1}{2}c_1(E(d))\cdot (c_1(E(d))-K_X) - c_2(E(d))\bigr\}.$\\

\hspace{2mm}

$ =  \dfrac{r}{2}\bigl[c_1(E(d))\cdot c_1(\mathcal{O}_X(-S))\bigr]m + \dfrac{r}{2}\bigl[(c_1(E(d))-K_X)\cdot c_1(\mathcal{O}_X(-S))\bigr]m $\\

$+ \dfrac{r^2}{2}\bigl[ c_1(\mathcal{O}_X(-S)\cdot c_1(\mathcal{O}_X(-S)\bigr]m 
 - (r-1)\bigl[ c_1(E(d))\cdot c_1(\mathcal{O}_X(-S))\bigr]m$\\
 
 $- \binom{r}{2}\bigl[c_1(\mathcal{O}_X(-S))\cdot c_1(\mathcal{O}_X(-S))\bigr]m  -r^2\bigl(\chi(\mathcal{O}_X)-1\bigr)\bigl\{c_1(\mathcal{O}_X(-S))\cdot A\bigr\}\\$
 
 $ - \dfrac{r}{2} c_1(E(d))\cdot \bigl(c_1(E(d))-K_X \bigr)\{c_1(\mathcal{O}_X(-S))\cdot A\bigr\} - rc_2(E(d))\{c_1(\mathcal{O}_X(-S))\cdot A\bigr\}.$

 \hspace{3mm}

 Now note that $$c_1(E(d))=c_1(E)+rc_1(\mathcal{O}_X(dD)) = \mathcal{O}_X\bigl((rd+1)D\bigr).$$

 Similarly,
 \vspace{2mm}

 $c_2(E(d)) = c_2(E)+(r-1)c_1(E)\cdot c_1(\mathcal{O}_X(dD))+\binom{r}{2}c_1(\mathcal{O}_X(dD)\cdot c_1(\mathcal{O}_X(dD))$\\

 $= c_2(E) + (r-1)d(D^2)+\binom{r}{2}d^2(D^2).$\\

 This implies that \\

 $\Bigr[\bigl\{c_1(E(d))\cdot A\bigr\}\bigl\{ h^0\bigl(E(d)\otimes \mathcal{O}_X(-S)\bigr)-r \bigr\}\Bigr] - \Bigr[\bigl\{c_1\bigl(E(d)\otimes \mathcal{O}_X(-S)\bigr)\cdot A\bigr\}\bigl\{ h^0(E(d))-r  \bigr\}\Bigr]$\\

 $=-r(rd+1)^2(D\cdot S)(D\cdot A)+\dfrac{r}{2}(rd+1)(K_X\cdot S)(D\cdot A) + \dfrac{r^2}{2}(rd+1)(S^2)(D\cdot A)$\\
 
 $+ (r-1)(rd+1)^2(D\cdot S)(D\cdot A) - \binom{r}{2}(rd+1)(S^2)(D\cdot A) +r^2\bigl(\chi(\mathcal{O}_X)-1\bigr)(S\cdot A) $\\
 
 $+ \dfrac{r}{2}(rd+1)^2(D^2)(S\cdot A) - \dfrac{r}{2}(rd+1)(D\cdot K_X)(S\cdot A)\\$
 
 $+  rc_2(E)(S\cdot A) + rd(r-1)(D^2)(S\cdot A) + r\binom{r}{2}d^2(D^2)(S\cdot A)\\$

 \vspace{2mm}

 $ = \Bigl[-(rd+1)^2(D\cdot S)(D\cdot A)+\dfrac{r}{2}(rd+1)(K_X\cdot S)(D\cdot A)+\dfrac{r}{2}(rd+1)(S^2)(D\cdot A)\\$
 
 $+r^2\bigl(\chi(\mathcal{O}_X)-1\bigr)(S\cdot A) + rc_2(E)(S\cdot A) + \dfrac{r}{2}\Bigl\{2r^2-r)d^2+(4r-2)d+1\Bigr\} (D^2)(S\cdot A)\\$ 
 
 $-\dfrac{r}{2}(rd+1)(D\cdot K_X)(S\cdot A)\Bigr]$\\

 \vspace{2mm}

 $=Ad^2+Bd+C$

 \vspace{2mm}

 where 
 
 \begin{align*}
     A & = -\,r^2 (D\cdot S)(D\cdot A)
\;+\;
\dfrac{r^2(2r-1)}{2}\,(D^2)(S\cdot A),\\
&= \dfrac{r^2}{2}\bigl[(2r-1)(D^2)(S\cdot A)  - 2(D\cdot A)(D\cdot S) \bigr]
 \end{align*}\\
\begin{align*}
B & = -\,2r\,(D\cdot S)(D\cdot A)
\;+\;
\frac{r^2}{2}\bigl((K_X\cdot S)+(S^2)\bigr)(D\cdot A)
\;+\;\\
&\hspace{2cm} r(2r-1)(D^2)(S\cdot A)
\;-\;
\frac{r^2}{2}(D\cdot K_X)(S\cdot A)\\
& =  r\Bigl[(2r-1)(S\cdot A) (D^2)-2(D\cdot A)(D\cdot S)\Bigr]+\dfrac{r^2}{2}\Bigl[\Bigl( (K_X\cdot S)+S^2\Bigr)(D\cdot A)-(D\cdot K_X)(S\cdot A)\Bigr]
\end{align*}

\vspace{2mm}
and 
\begin{align*}
C&= -\,(D\cdot S)(D\cdot A)
\;+\;
\dfrac{r}{2}(K_X\cdot S)(D\cdot A)
\;+\;
\dfrac{r}{2}(S^2)(D\cdot A)
\;+\;
r^2\bigl(\chi(\mathcal O_X)-1\bigr)(S\cdot A)
\;+\;\\
& \hspace{3cm} r\,c_2(E)(S\cdot A)
\;+\;
\dfrac{r}{2}(D^2)(S\cdot A)
\;-\;
\dfrac{r}{2}(D\cdot K_X)(S\cdot A),\\
&= \Bigl[\dfrac{r}{2}(S\cdot A)(D^2) -(D\cdot A)(D\cdot S)\Bigr]+\dfrac{r}{2}\Bigl[\Bigl(( K_X\cdot S)+S^2\Bigr)(D\cdot A)-(D\cdot K_X)(S\cdot A)\Bigr] + \\
&
\hspace{3cm} r^2\bigl( \chi(\mathcal{O}_X)-1\bigr) + r\,c_2(E)(S\cdot A)
\end{align*}

\vspace{5mm}

For $d\gg0$, the coefficient of $d^2$ will dominate provided $$r^2(r-\dfrac{1}{2})(D^2)\cdot (S\cdot A)-r^2(D\cdot S)(D\cdot A) > 0$$

i.e. when $$(2r-1)(D^2)(S\cdot A)  - 2(D\cdot A)(D\cdot S) >0.$$

In such case $$\mu_A\bigl(M_{E(d)\otimes \mathcal{O}_X(-S)}\bigr) > \mu_A\bigl(M_{E(d)}\bigr)$$ for $d\gg0$.

Also, if $A\geq 0$ and $d\geq -\dfrac{B}{C},$ we have $Ad^2+Bc+C > 0$ and thus, in this case also, $$\mu_A\bigl(M_{E(d)\otimes \mathcal{O}_X(-S)}\bigr) > \mu_A\bigl(M_{E(d)}\bigr).$$

The following result is inspired by  \cite[Proposition 8]{DHJS25}.
\begin{thm}
    Let $X$ be a smooth surface of Picard number at least 3 having an effective cone generated by curves $C_1,C_2,\cdots, C_m$ with mutual intersection multiplicity of at most 1. Then for every vector bundle $E$ with its determinant bundle $D=\det(E)$ ample, there exists an ample polarisation $A$ such that for $d\gg 0$ the syzygy bundle $M_{E(d)}$ is not stable with respect to $A$.
    \begin{proof}
        Let $j$ be index such that $D\cdot C_j \leq D\cdot C_i$ for every $1\leq i \leq m.$ Let 
        \begin{center}
       $t_0 = \max\{ t \mid D-tC_j$ is nef $\}$ 
       \end{center}
       and we set $A=D-t_0C_j.$ We then have that $A$ is on the boundary of the nef cone of $X$, and $A\cdot C_i=0$ for some generator $C_i$. By the hypothesis $C_i\cdot C_j \leq 1$ and $D$ ample. This shows that $D\cdot C_i=t_0.$ By the choice of $C_j,$ we then have $t_0 \geq D \cdot C_j.$ We can also assume that the curves $C_1.C_2,\cdots,C_m$ have negative self-intersection.

       Now letting $S=C_j$ in Theorem \ref{thm-main}(1), and by a similar argument as in \cite[Proposition 8]{DHJS25} we get  \begin{align*}
        (D^2)(S\cdot A)  - 2(D\cdot A)(D\cdot S)
       &=(C_j\cdot(D-t_0C_j))D^2-2(D\cdot(D-t_0C_j))(D\cdot E_j)\\
       &>-D^2(D\cdot C_j)-t_oC_j^2D^2\\
       &\geq -D^2(D\cdot C_j) + t_0D^2\\
       &=D^2(t_0-D\cdot C_j) \geq 0.
        \end{align*}
        As $A$ and $D$ are ample, and $S$ is a curve, we conclude 
        $$(2r-1)(D^2)(S\cdot A)  - 2(D\cdot A)(D\cdot S) \geq 0.$$
       
       Now, note that, $A$ is only nef, and not ample. However, replacing $A$ by $A' = A+\varepsilon C_j$ will be ample for $0< \varepsilon \leq 1$ and by continuity, we conclude $$(2r-1)(D^2)(S\cdot A')  - 2(D\cdot A')(D\cdot S) \geq 0$$
       for sufficiently small $\varepsilon$. 

       Now we choose $d$ sufficiently large so that $E(d)$ and $E(d)\otimes \mathcal{O}_X(-S)$ are globally generated. Then Lemma \ref{lem1} implies that $M_{E(d)\otimes \mathcal{O}_X(-S)}$ is a subbundle of $M_{E(d)}$, and by Theorem \ref{thm-main} (1), we conclude that $M_{E(d)\otimes \mathcal{O}_X(-S)}$ is a destabilizing subbundle of $M_{E(d)}$. This shows that $M_{E(d)}$ is unstable with respect to $A$.
        
    \end{proof}
\end{thm}

\end{document}